\documentclass[11pt]{amsart}

\usepackage{amssymb, amsmath}
 \usepackage[foot]{amsaddr}
\usepackage[alwaysadjust]{paralist}
\usepackage[capitalize]{cleveref}
\numberwithin{equation}{section}

\newtheorem{cor}[equation]{Corollary}
\newtheorem{lemma}[equation]{Lemma}
\newtheorem{prop}[equation]{Proposition}
\newtheorem{theorem}[equation]{Theorem}
\theoremstyle{definition}
\newtheorem{defn}[equation]{Definition}

\newtheorem{rem}[equation]{Remark}

\def\IN{\mathbb N}
\def\IR{\mathbb R}

\def\IZ{\mathbb Z}

\def\eps{\varepsilon}

\newcommand{\divergence}{\operatorname{div}}

\newcommand{\supp}{\operatorname{supp}}

\newcommand{\vol}{\operatorname{vol}}

\title[Extremal metrics]{Extremal metrics for Laplace eigenvalues in perturbed conformal classes on products}
\author{Henrik Matthiesen}
\address
{Max Planck Institute for Mathematics,
Vivatsgasse 7, 53111 Bonn.}
\email{hematt\@@mpim-bonn.mpg.de}
\date{\today}

\thanks{\emph{Acknowledgments.}
It is a pleasure
to thank my advisor Werner Ballmann for permanent support.
I gratefully acknowledge the support and hospitality of
the Max Planck Institute for Mathematics in Bonn.
}

\begin{document}

\maketitle

\begin{abstract}
In this short note, we prove that conformal classes which are small perturbations of a product conformal class on a product with a standard sphere admit a metric extremal for
some Laplace eigenvalue.
As part of the arguments we obtain perturbed harmonic maps with constant density.

\end{abstract}

\section{Introduction}

For a closed manifold $M$ we are interested in the eigenvalues of the Laplace operator considered as functionals of the metric.

We denote by
\begin{align*}
	\mathcal{R}:=\{g \ : \ g \ \text{is a Riemannian metric on}\ M  \text{ with}\ \vol(M,g)=1\}
\end{align*}
the space of all unit volume Riemannian metrics on $M$
endowed with the $C^\infty$-topology, i.e.\ the smallest topology containing any $C^k$-topology.
The group $C^\infty_+(M)$ of positive smooth functions acts via (normalized) pointwise multiplication on $\mathcal{R},$
\begin{equation}
\phi. g := \vol(M,\phi g)^{-2/n} \phi g,
\end{equation}
so that $\vol(M,\phi.g)=1.$
The quotient space
\begin{align*}
\mathcal{C}= C^\infty_+(M) \backslash \mathcal{R}
\end{align*} 
 is the space of all conformal structures on $M.$

Since $M$ is compact, the spectrum of $\Delta_g$ consists of eigenvalues of finite multiplicity only for any $g \in \mathcal{R}.$
We list these as
\begin{equation}
0=\lambda_0<\lambda_1\leq\lambda_2\leq \dots,
\end{equation}
where we repeat an eigenvalue as often as its multiplicity requires.

In recent years there has been much interest in finding extremal metrics for eigenvalues $\lambda_k$ considered either as functionals 
\begin{equation} \label{critical_1}
\lambda_k \colon \mathcal{R} \to \IR
\end{equation}
or 
\begin{equation} \label{critical_2}
\lambda_k \colon [g] \to \IR,
\end{equation}
where 
\begin{equation*}
[g]=\{\phi g \ :\ \phi \in C_+^\infty(M)\}
\end{equation*}
denotes the conformal class of a metric g,
see for instance  \cite{ElSoufi_Giacomini_Jazar, Fraser_Schoen, Kokarev, Nadirashvili_1996, Petrides_2}, and references therein.
These functionals will not be smooth but only Lipschitz, therefore extremality has to be defined in an appropriate way, see below.

One reason to study these extremal metrics is their intimate connection to other classical
objects from differential geometry.
For \eqref{critical_1}, these are minimal surfaces in spheres, and for \eqref{critical_2} these are sphere-valued harmonic maps with constant density, so called \emph{eigenmaps}.
There has been a lot of effort in the past to understand, which manifolds admit eigenmaps
or even minimal isometric immersions into spheres, see for instance \cite[Chap. 6]{Urakawa} for a general overview over classical results for eigenmaps including the generalized Do Carmo--Wallach theorem, and \cite{Bryant, Lawson} to mention only the two most classical results.

Before we state our results, we have to introduce some notation.
Let $M$ be a smooth, closed manifold.

A smooth map $u \colon M \to S^\ell$ is called an \emph{eigenmap}, if it is harmonic, i.e.
\begin{equation} \label{eigen_eqn}
\Delta u = |\nabla u|^2 u,
\end{equation}
and has constant density $|\nabla u|^2=const.$
In other words, the components of $u$ are all eigenfunctions corresponding to 
the same eigenvalue.
Note that most Riemannian manifolds do not admit eigenmaps, since the spectrum
is generically simple by \cite[Theorem 8]{Uhlenbeck}.
Even more, the spectrum of a generic metric in a conformal class is simple \cite{BW, GLSD,Uhlenbeck}.
Moreover, we would like to point out that it is not clear at all whether  eigenmaps exist in the presence of large multiplicty.

\begin{theorem} \label{main}
Let $(M,g)$ be a closed Riemannian manifold of dimension $\dim(M) \geq 3$, and assume
\begin{itemize}
\item[(i)] There is a a non-constant eigenmap $u \colon (M,g) \to S^1,$
\\
or
\item[(ii)] $(M,g)=(N \times S^\ell, g_N+ g_{st.}),$ where $g_{st.}$ denotes the round metric of curvature $1$ on $S^\ell.$
\end{itemize}
Then there is a neighbourhood $U$ of $[g]$ in $\mathcal{C},$ such that for any
$c \in U,$ there is a representative $h \in c,$ such that
$(M,h)$ admits a non-constant eigenmap to $S^1$ respectively $S^\ell.$
\end{theorem}

An obvious question is then, whether the set of conformal structures admitting non-constant eigenmaps is
always non-empty.
We answer this at least in the following case.

\begin{cor} \label{mappingtori}
Assume $\phi \colon M \to S^1$ is a submersion.
Then the set $\mathcal{E} \subset \mathcal{C}$
of conformal structures admitting non-trivial eigenmaps to $S^1$
is open and non-empty.
\end{cor}

\begin{rem}
It is not clear, whether $\mathcal{E}$ is also closed.
This question is related to possible degenerations of $n$-harmonic maps, as it will become
clear from the proof.
\end{rem}

Not every manifold admits a submersion to $S^1.$
In fact, there are topological obstructions to the existence of such a map.

More precisely, since $S^1$ is a $K(\IZ,1)$, a submersion gives rise to a non-trivial element
in $H^1(M,\IZ).$
Moreover, the differentials of local lifts of the submersion to $\IR,$ give rise
to a globally defined nowhere vanishing $1$-form.
In particular, $M$ needs to have $\chi(M)=0.$

As mentioned above, the existence of an eigenmap $u\colon (M,h) \to S^\ell$ for a metric $h \in [g]$ implies
that $h$ is extremal for some of the functionals $\lambda_k$ on $[g]$.
Therefore, \cref{main} and \cref{mappingtori} have the following consequences for the existence of extremal metrics.

\begin{cor} \label{crit_1}
Under the assumptions of \cref{main}, there is a neighbourhood $U$ of $[g]$ in $\mathcal{C},$ such that
for any $c \in U,$ there is a representative $h \in c,$ such that $(M,h)$
is extremal for some eigenvalue functional on $c.$
\end{cor}

\begin{cor} \label{crit_2}
Under the assumptions of \cref{mappingtori}, the set $\mathcal{E} \subset \mathcal{C}$
of conformal structures admitting extremal metrics for some eigenvalue functional on conformal classes
is open and non-empty.
\end{cor}

The proof of \cref{main} is rather simple once the correct conformally invariant formulation of the assertion is found.

This is as follows. 
Let $n$ be the dimension of $M.$
Then a smooth map into a sphere is called \emph{$n$-harmonic}, if it is a critical point 
of the \emph{$n$-energy}
\begin{align*}
E_n[u]= \int_M |du|^n dV_g,
\end{align*}
which is a conformally invariant functional.
These are precisely the solutions of the equation
\begin{align} \label{n-harmonic-eqn}
-\divergence(|\nabla u|^{n-2} \nabla u)= |\nabla u|^n u.
\end{align}
From \eqref{eigen_eqn} and \eqref{n-harmonic-eqn} it is evident, that
an eigenmap defines an $n$-harmonic map, which has $\nabla u \neq 0$ everywhere.
The crucial observation is that also the converse holds up to changing the metric conformally, see \cref{n-harmonic-eigenmap}.

Therefore, we will be concerned with $n$-harmonic maps with nowhere vanishing derivative.

In order to deduce \cref{mappingtori} from \cref{main}, it suffices to find a single non-trivial eigenmap
$u \colon (M,g) \to S^1$ for some metric $g.$
This turns out to be very easy using that $M$ is a mapping torus.

In  \cref{tools} we discuss the necessary preliminaries on $n$-harmonic maps and Laplace eigenvalues.
\cref{proofs} contains the proofs.

\section{Preliminaries} \label{tools}

First, we explain the notion of extremal metrics and its connection to eigenmaps.

\subsection{Extremal metrics for eigenvalue functionals}

In presence of multiplicity, the functionals $\lambda_k$ are not differentiable, but only Lipschitz.
However, it turns out that for any analytic deformation, left and right derivatives exist.
Using this El Soufi--Ilias introduced a notion of extremal metrics for these functionals.

\begin{defn}[{\cite[Definition 4.1]{ElSoufi_Ilias}}] \label{def_extremal}
A metric $g$ is called \emph{extremal} for the functional $\lambda_k$ restricted to the conformal class
$[g]$ of $g,$ if for any analytic family of metrics $(g_t) \subset [g],$
with $g_0=g,$ and $\vol(M,g_0)=\vol(M,g_t),$ we have
\begin{align*}
\left. \frac{d}{dt} \right|_{t=0^-} \lambda_k(g_t) \cdot \left. \frac{d}{dt} \right|_{t=0+} \lambda_k(g_t) \leq 0.
\end{align*}
 \end{defn}

We have

\begin{theorem}[{\cite[Theorem 4.1]{ElSoufi_Ilias}}] \label{extremal}
The metric $g$ is extremal for some eigenvalue $\lambda_k$ on $[g]$ if and only if
there is a eigenmap $u \colon (M,g) \to S^\ell$ given by $\lambda_k(g)$-eigenfunctions and either
$\lambda_{k-1}(g)<\lambda_k(g)$, or $\lambda_k(g) < \lambda_{k+1}(g).$
\end{theorem}

\subsection{Background on $n$-harmonic maps}

First of all we need some background on the existence of $n$-harmonic maps.
We call a map $u \in W^{1,n}(M,S^\ell)$ \emph{weakly $n$-harmonic,}
if it is a weak solution of
\begin{equation}
-\divergence(|\nabla u|^{n-2}\nabla u) = |\nabla u|^n u.
\end{equation}

We assume that we have fixed a CW-structure on $M,$ and denote by $M^{(l)}$ its $l$-skeleton.
Let $v \colon M \to S^\ell$ be a Lipschitz map, where $l<n=\dim M.$
Denote by $v^{(l)}$ the restriction of $v$ to the $l$-skeleton of $M.$
The \emph{$l$-homotopy type of $v$} is the homotopy type of $v^{(l)}.$

\begin{theorem}[{\cite[Theorem 3.4]{White}}]  \label{existence_2}
There exists a weakly $n$-harmonic map $u \colon M \to S^\ell,$ with well-defined $l$-homotopy type,
which agrees with the $l$-homotopy type of $v.$
Moreover, $u$ minimizes the $n$-energy among all such maps.
\end{theorem}

We do not elaborate here on how the $l$-homotopy type is defined for maps in $W^{1,n}(M,S^\ell).$
For our purposes this is not necessary, since the map $u$ is actually continuous.

\begin{theorem} \label{regularity}
Let $u \in W^{1,n}(M,S^\ell)$ be a weakly $n$-harmonic map, which is a minimizer for its own $l$-homotopy type.
There is a constant $C$ depending on an upper bound on the $n$-energy of $u,$ and on the bounds of the sectional curvature and injectivity radius of $M,$ such that
$\|u \|_{C^{1,\alpha}} \leq C.$ 
\end{theorem}

\begin{proof}
Let $x \in M$ and $r>0$ be small enough.
If  $v \in W^{1,n}(B(x,r),S^\ell)$ with $u=v$ on $\partial B(x,r)$, we can consider the map $w \in W^{1,n}(M,S^\ell)$ given by
$u$ in $M \setminus B(x,r)$ and by $v$ in $B(x,r)$.
It is shown in \cite[Theorem 2.8]{Pigola_Veronelli} that the $l$-homotopy type of $w$ agrees with the $l$-homotopy type of $u$.
In particular, we need to have
\begin{equation*}
\int_{B(x,r)} |du|^n dV_g \leq \int_{B(x,r)} |dv|^n dV_g,
\end{equation*}
which means that $u$  is a \emph{minimizing} $n$-harmonic map.
Therefore, the assertion follows e.g.\ from \cite[Theorem 2.19]{NVV}.
\end{proof}

In particular, these estimates are uniform as $g$ varies over a compact set of $\mathcal{R},$ as long as the energy stays bounded.

At points, in which we do not have a lack of ellipticity, we actually get higher regularity.

\begin{theorem} \label{higher_regularity}
A weakly $n$-harmonic map $u \in C^{1,\alpha}$ is smooth near points with $\nabla u \neq 0.$ 
\end{theorem}

This follows from standard techniques for quasilinear elliptic equations. 
For completeness, we give a proof in \cref{ssec_reg}.

The main reason for the restrictive assumptions in item $(ii)$ of \cref{main} is that the above results
do not imply that for a sequence $g_k \to g$  we can find a sequence of  $n$-harmonic maps $u_k$ (w.r.t.\ $g_k$),
such that $u_k \to u,$ for a given $n$-harmonic map $u.$

In the case of maps to the circle, this problem does not appear, thanks to

\begin{theorem}[{\cite[Theorem A]{Veronelli}}] \label{uniqueness}
Up to rotations of $S^1,$ $n$-harmonic maps $u \colon M \to S^1$ are unique in their homotopy class.
\end{theorem}

\section{Proofs} \label{proofs}

\subsection{Higher regularity of $n$-harmonic maps} \label{ssec_reg}
In this section we give a proof of \cref{higher_regularity}.
We start with $W^{2,2}$-regularity.
The proof follows using standard techniques, since under our assumptions the equation is of the form
\begin{equation}
- (L u)(x) - b(x) u(x)=0,
\end{equation}
with $L$ a quasilinear operator, which is elliptic at $u$ (as demonstrated in \cref{elliptic} below) and $b \in L^\infty.$

\begin{lemma}
Let $U \subset M$ be open and $u \colon (U,g) \to S^\ell$ be weakly $n$-harmonic.
Assume that $u \in C^{1,\alpha}(U,S^\ell)$ with $\nabla u \neq 0$ everywhere in $u$.
Then we have $u \in W^{2,2}_{loc}(U,S^\ell).$
\end{lemma}

\begin{proof}
For simplicity, we focus on the case $g_{ij}=\delta_{ij}$ 
and denote the usual differential of $u$ in Euclidean Space by $Du.$
The general case follows along the same lines but with some more notation.

Take open subsets $W\subset \subset V \subset \subset U,$ and a cut-off function $\eta$  which is $1$ in $W,$ and has $\supp \eta \subset V.$ 
We show that $u \in W^{2,2}(W,S^\ell).$
We use the test functions given by $\phi^k=-D_s^{-h}(\eta^2D_s^h u^k),$
where $D_s^h$ denotes the difference quotient operator in coordinate direction $s$,
\begin{equation}
D_s^h \phi (x)=\frac{1}{h}(\phi(x+h e_s)-\phi(x)).
\end{equation}
To handle notation, let us write
\begin{equation}
F^\alpha_k(D u)= |D u|^{n-2} \partial_\alpha u^k,
\end{equation}
and
\begin{equation}
G_k(u,D u) = |D u|^n u^k.
\end{equation}
Then we have
\begin{equation} \label{eqn}
-  \int_U F^\alpha_k(D u) \partial_\alpha D_s^{-h}(\eta^2 D_s^h u^k) = -  \int_U G_k(u, Du) D_s^{-h}(\eta^2 D_s^h u^k),
\end{equation}
Note that this is well-defined thanks to H{\"o}lder's inequality.
For the left hand side of \eqref{eqn}, we have
\begin{equation} \label{int_part}
-  \int_U F_k^\alpha(D u) \partial_\alpha D_s^{-h}(\eta^2 D_s^h u^k) = \int_U D_s^h (F^\alpha_k(D u)) \partial_\alpha (\eta^2 D_s^h u^k).
\end{equation}
We can write
\begin{equation}
\begin{split}
D_s^h F^\alpha_k(D u) & = \frac{1}{h} \int_0^1 \frac{d}{dt} F^\alpha_k(D u + t h D_s^h D u)dt
\\
& = \frac{1}{h} \int_0^1 	\frac{\partial F^\alpha_k}{\partial q_\beta^l} (D u + t h D_s^h D u)hD_s^h \partial_\beta u^l dt
\\
& =\int_0^1 	\frac{\partial F^\alpha_k}{\partial q_\beta^l} (D u + t h D_s^h D u)dt\, D_s^h \partial_\beta u^l dt
\\
& =: \theta_{kl}^{\alpha \beta}(D u) D_s^h \partial_\beta u^l.
\end{split}
\end{equation}
Note that this is well defined pointwise, since $u \in C^{1,\alpha}.$
The condition $|Du|\geq c >0,$ implies that $\theta^{\alpha \beta}_{kl}$ are uniformly super strongly elliptic for $h \ll 1,$ as demonstrated below.
Since the coefficients $\theta$ are uniformly super strongly elliptic, we have
\begin{equation}
\int_U \eta^2 |D_s^h Du|^2 \leq C  \int_U \eta^2 \theta_{kl}^{\alpha \beta}(D u) (D^h_s\partial_\alpha u^k)  (D_s^h \partial_\beta u^l).
\end{equation}
Moreover, since $\theta$ and $|D \eta|$ are bounded, we can estimate
\begin{equation}
\begin{split}
\left| \int_U \theta_{kl}^{\alpha \beta} (D_s^h \partial_\beta u^l) (D_s^h u^k) \eta \partial_\alpha \eta \right|
& \leq C \int_U |D_s^h Du| |D_s^h u|\eta
\\
& \leq C \eps \int_U \eta^2 |D_s^h Du|^2 + \frac{C}{\eps} \int_{V} |D_s^h u|^2
\\
& \leq C \eps \int_U \eta^2 |D_s^h Du|^2  + \frac{C}{\eps} \int_U |Du|^2,
\end{split}
\end{equation}
where we have used Young's inequality and $u \in W^{1,2}.$
Combining the last two estimates with \eqref{eqn} and \eqref{int_part}, we find that we can choose $\eps$ sufficiently small so that
\begin{equation}
\int_U \eta^2 |D_s^h Du|^2 \leq C \int_U |Du|^2 + C \left| \int_U G_k(u, Du) D_s^{-h}(\eta^2 D_s^h u^k)\right|.
\end{equation}
To estimate the last summand above, we note that $u,|D u| \in L^\infty,$ implies $G_k(u,D u) \in L^\infty,$ hence
\begin{equation}
\begin{split}
\left| \int_U G_k(u, Du) D_s^{-h}(\eta^2 D_s^h u^k)  \right|
& \leq C \int_U |D_s^{-h} (\eta^2 D_s^h u^k)|
\\
& \leq \frac{C}{\eps} \vol(U) + C \eps \int_U |D_s^{-h}(\eta^2 D_s^h u^k)|^2
\\
& \leq \frac{C}{\eps} + C \eps \ \int_U |D(\eta^2 D_s^h u^k)|^2
\\
& \leq \frac{C}{\eps}(1 + \int_U |Du|^2) + C \eps \int_U \eta^2 |D_s^h Du^k|^2.
\end{split}
\end{equation}
For $\eps$ sufficiently small, we can absorb the last term, and find
\begin{equation}
\int_{V} |D_s^h Du|^2 \leq \int_U \eta^2 |D_s^h Du|^2 \leq \frac{C}{\eps}(1 + \int_U |Du|^2).
\end{equation}
Thus $u \in W^{2,2}_{loc}(U,S^\ell).$
\end{proof}

We still need to justify that the coefficients $\theta^{\alpha \beta}_{kl}$ are uniformly super strongly elliptic.
\begin{lemma} \label{elliptic}
There is $h_0>0$ depending on $\| Du \|_{C^{0,\alpha}}$ such that we have
$\theta_{kl}^{\alpha \beta} A^k_\alpha A^l_\beta \geq \nu |A|^2,$ for any $h$ with $|h| \leq h_0$ and $\nu=\nu(c),$ 
where $|Du|^2 \geq c.$
\end{lemma}

\begin{proof}
We have
\begin{equation}
\frac{\partial F_k^\alpha}{\partial q_\beta^l}(q)=|q|^{n-4} (|q|^2 \delta^{\alpha \beta}\delta_{kl} + (n-2) q_\alpha^k q_\beta^l).
\end{equation}
Thus, it is not very hard to see that
\begin{equation}
\begin{split}
\frac{\partial F_k^\alpha}{\partial q_\beta^l}(q) A_\alpha^k A_\beta^l 
& = |q|^{n-4} (|q|^2 \delta^{\alpha \beta}\delta_{kl} A_\alpha^k A_\beta^l + (n-2) q_\alpha^k q_\beta^l A_\alpha^k A_\beta^l)
\\
& \geq |q|^{n-2}|A|^2 
\\
& \geq 2 \nu |A|^2,
\end{split}
\end{equation}
as long as $|q|^2 \geq (2 \nu)^{2/(n-2)}.$
Since $Du \in C^{0,\alpha},$ and $|D u|^2 \geq c,$ we can choose $h_0 \ll 1,$ such that 
$|(1-t)D u (x+he_s) + t D u(x)|^2 \geq c/2,$
for all $x,$ and $|h| \leq h_0.$
Clearly, this implies
\begin{equation*}
\begin{split}
\theta_{kl}^{\alpha \beta} (D u)(x) A_\alpha^k A_\beta^l 
& = \int_0^1 \frac{\partial F_k^\alpha}{\partial q_\beta^l}((1-t)D u (x+he_s) + t D u(x)) A_\alpha^k A_\beta^l dt
\\
& \geq  \int_0^1 \nu |A|^2 dt 
\\
& \geq  \nu |A|^2,
\end{split}
\end{equation*}
for $\nu = c^{(n-2)/2}/2$.
\end{proof}

In the next step we derive the equation for $\partial_\alpha u^k$ and apply Schauder estimates to gain higher regularity.
In particular, this completes the proof of \cref{higher_regularity}

\begin{lemma}
Under the above assumptions, the function $u$ is smooth.
\end{lemma}

\begin{proof}
Write
\begin{equation}
\vartheta_{kl}^{\alpha \beta} = \frac{\partial F^\alpha_k}{\partial q_\beta^l}.
\end{equation}
By the calculation above, these coefficients are uniformly super strongly elliptic
at $u.$
We test the equation for $u^k$ with $\partial_\alpha \phi^k$ for some test function $\phi$ and integrate by parts in order to find
\begin{equation}
\int_U \vartheta_{kl}^{\alpha \beta} (\nabla u) \partial_{\beta \gamma} u^l \partial_\alpha \phi^k = \int_U \partial_\gamma G_k(u, \nabla u) \phi^k
\end{equation}
In other words, $v=\partial_\gamma u$ is a weak solution to
\begin{equation}
-\divergence (\vartheta(Du) v) = \partial_\gamma G(u,Du).
\end{equation}
Since $|D u|^2 \geq c >0,$ the right hand side of this equation is in $C^{k,\alpha},$
once we have $u \in C^{k+1,\alpha}.$
In this case the left hand side has coefficients in $C^{k,\alpha},$ thus it follows that
$v \in C^{k+1,\alpha}$ and thus $u \in C^{k+2, \alpha}.$
Since we know $u \in C^{1,\alpha},$ we can start this bootstrap argument at $k=0,$ and get $u \in C^{\infty}.$
\end{proof}

\subsection{Proofs of main results}

We start with the following simple but crucial observation.

\begin{lemma} \label{n-harmonic-eigenmap}
Let $u \colon (M,g) \to S^\ell$ be a smooth $n$-harmonic map with $du \neq 0$ everyhwere.
Then there is metric $g'$ conformal to $g,$
such that $u \colon (M,g') \to S^\ell$ is an eigenmap.
\end{lemma}

\begin{proof}
Define $g'=|du|_g^2 g.$
Since we assumed $du \neq 0$ everywhere, this defines a smooth metric, which is  conformal
to $g.$
Then $|du|_{g'}^2=|du_g|^{-2} |du|_g^2=1.$
Finally, $u$ solves
\begin{align*}
-\divergence_g (|du|_g^{n-2} \nabla u)= |du|_g^n u,
\end{align*}
which can also be written as
\begin{align*}
\Delta_{g'} u = -\frac{1}{|du|_g^n}\divergence_g (|du|_g^{n-2} \nabla u)= u,
\end{align*}
hence $u \colon (M,g') \to S^\ell$ is an eigenmap.
\end{proof}

In order to prove \cref{main} it now suffices to show that metrics close to the initial metric $g$ on $M$ also
admit smooth $n$-harmonic maps with nowhere vanishing derivative.

\begin{proof}[Proof of \cref{main} (i)]
Let $ u \colon (M,g) \to S^1$ be an eigenmap and
assume that the assertion of the theorem was not correct.
This means that any neighbourhood $U \subset \mathcal{C}$ of $[g]$ contains a conformal class which does not contain any representative which admits an eigenmap to $S^1$.
Let $U_k\subset \mathcal{R}$ be a sequence of open neighbourhoods of $g$ with $\cap_{k \in \IN} U_k = \{g\}$.
(Such a sequence exists since the $C^\infty$-topology on $\mathcal{R}$ is first countable and Hausdorff.)
Denote by $\pi \colon \mathcal{R} \to \mathcal{C}$ the quotient map and observe that this is an open map.
In particular, the sets $\pi(U_k) \subset \mathcal{C}$ are open and
we can find $g_k \in U_k$ such that no metric in $[g_k]$ admits an eigenmap to $S^1$.
By \cref{n-harmonic-eigenmap} this implies that $g_k$ itself cannot admit a nowhere
vanishing $n$-harmonic map to $S^1$.

We now plan to use \cref{existence_2}  to obtain weakly $n$-harmonic maps $u_k \colon (M,g_k) \to S^1$ which are close to $u$ for $k$ sufficiently large.
By assumption the $u_k$ have some point $x_k$ with $du_k(x_k)=0$.
This forces $u$ to have a critical point as well, which gives the desired contradiction.

More precisely, we apply \cref{existence_2} to $u \colon (M,g_k) \to S^1$ and obtain $n$-harmonic representatives $u_k \colon (M,g_k) \to S^1$ of $[u]$.
If $du_k \neq 0$ everywhere, \cref{higher_regularity} implies that $u_k$ is a smooth $n$-harmonic map from $(M,g_k)$ to $S^1$ with nowhere vanishing derivative contradicting the construction of $g_k$ in the preceding paragraph.
Therefore, we can find $x_k \in M$ such that $du_k(x_k)=0$.
Since $\dim(M) \geq 3$ and $S^1 \simeq K(\IZ,1),$ we have that $w \simeq u$ if and only if their $l$-homotopy type agrees for some $l \geq 2.$
In particular, we have that
\begin{equation*}
\int_M |du_k|^n dV_{g_k} \leq \int_M |du|^n dV_{g_k}
\leq C \int_M |du|^n dV_g,
\end{equation*}
so that we are in the position to apply \cref{regularity}.

By taking a subsequence if necessary, we may assume that $x_k \to x.$
Thanks to \cref{regularity} and the compact embedding $C^{1,\alpha}(M) \hookrightarrow C^{1,\beta}(M)$ for $\beta<\alpha, $ we can extract a further
subsequence, such that $u_k \to v$ in $C^{1,\beta}(M,g).$
We have 
\begin{align*}
\int_M |dv|_g dV_g & = \lim_{k \to \infty} \int_M |dv|_{g_k} dV_{g_k}
\\
& \leq \lim_{k \to \infty} \left( \int_M |du_k|_{g_k} dV_{g_k} + \int_M \left| |dv|_{g_k} - |du_k|_{g_k} \right| dV_{g_k} \right)
\\
& \leq  \lim_{k \to \infty}\left( \int_M |dw|_{g_k} dV_{g_k} + C d_{C^{1, \beta}(M,g_k)} (v,u_k) \right)
\\ 
& \leq \lim_{k \to \infty} \int_M |dw|_{g_k} dV_{g_k} + \lim_{k \to \infty} C d_{C^{1, \beta}(M,g)} (v,u_k)
\\
& = \int_M |dw|_g dV_g,
\end{align*}
for any $w \simeq u.$
It follows, that $v$ is $n$-harmonic and homotopic to $u.$
Thus it follows from \cref{uniqueness} that there is $A \in SO(2),$ such that
$A\circ v=u.$
Then $A \circ u_k \to u$ in $C^{1,\beta}(M).$
It follows, that
\begin{align*}
 |du(x)| \leq \limsup_{k \to \infty} C d(x,x_k)^\beta =0,
\end{align*}
contradicting the assumption on $u.$
\end{proof}

In order to adapt the strategy from above for more general situations,
we need to understand whether there exist eigenmaps $u \colon (M,g) \to S^\ell,$
which can be approximated through $n$-harmonic maps for any sequence of metrics
$g_k \to g.$

This is precisely what we do now for product metrics $g_{st.}+g_N$ on $S^\ell \times N.$
The natural candidate here is the projection map onto $S^\ell.$
In what follows $n$ will denote the dimension of $N,$ so that the dimension of $N \times S^\ell$ is $n+l.$

\begin{prop} \label{prop_prod_est}
Let $g=g_N + g_{st.}$ be a product metric on $N \times S^\ell,$
with $g_{st.}$ the round metric of curvature $1$ on $S^\ell.$
The projection $u \colon N \times S^\ell \to S^\ell$ onto the second factor
is the unique minimizer for the $(n+l)$-energy in its $l$-homotopy class up to rotations of $S^\ell.$
\end{prop}

\begin{proof}
Let $v \colon N \times S^\ell \to S^\ell$ be a Lipschitz map whose restriction to the $l$-skeleton of $N \times S^\ell$ is homotopic to the restriction of the
projection $N \times S^\ell \to S^\ell$ to the $\ell$-skeleton.
We want to estimate
\begin{equation}
 \int_{N \times S^\ell} |dv|_g^{n+\ell} dV_g
\end{equation}
from below.

We have
\begin{equation} \label{prod_est}
\begin{split}
 \int_{N \times S^\ell} &|dv|_g^{n+\ell} dV_g  
 \\
 	& = \int_N \int_{S^\ell} (|\nabla^N v|^2 + |\nabla^{S^\ell} v|^2)^{(n+\ell)/2} (x, \theta) d\theta dx
 	\\ 
 	& \geq \int_N \int_{S^\ell} |\nabla^{S^\ell}v|^{n+\ell}(x, \theta)  d\theta dx 
 	\\
 	& \geq (\ell+1) \omega_{\ell+1})^{-n/\ell} \int_N \left( \int_{S^\ell} |\nabla^{S^\ell}v|^\ell(x , \theta) d \theta \right)^{(n+\ell)/\ell} dx,
 \end{split}
\end{equation}
where we have used H{\"o}lder's inequality in the last step.
Equality holds in the above inequalities if and only if $|\nabla^M v|^2=0$ and $|\nabla^{S^\ell}v|^2= const.$

In order to estimate the remaining integral in the last line of \eqref{prod_est} we use that the maps $v(x, \cdot) \colon S^\ell \to S^\ell$ have degree $1.$
This can be seen by inspecting the $l$-homotopy type of $v$:
If we endow $S^\ell$ with the CW-structure consisting of a single $0$- and a single $\ell$-cell,
we have
$( N \times S^\ell)^{(\ell)} = N^{(\ell)} \times \{\theta_0\} \cup \{x_0\} \times S^\ell=N^{(\ell)} \vee S^\ell$
with $\theta_0 \in S^\ell$ and $x_0 \in N$ corresponding to the $0$-cells.
The projection onto $S^\ell$ restricts to the map $N^{(\ell)} \vee S^\ell \to S^\ell$ that collapses the first summand and is the identity on $S^\ell.$
In particular, we find that for any $v,$ such that $v^{(\ell)}$ is homotopic to the map described above, the degree of $v(x_0,\cdot) \colon S^\ell \to S^\ell$ equals $1.$
Since $N$ is connected, $v(x,\cdot) \simeq v(x_0,\cdot)$  for any $x,$ thus
$\deg v(x,\cdot)=1$ for any $x \in N.$

This implies, that
\begin{equation} \label{deg_est}
 \int_{S^\ell} |\nabla^{S^\ell}v|^l(x , \theta) d \theta \geq (\ell+1)\omega_{\ell+1}\ell^{\ell/2} |\deg v(x, \cdot) |
 = (\ell+1)\omega_{\ell+1}l^{\ell/2}.
\end{equation}
Here, equality holds if and only if $v(x, \cdot)$ is conformal.
Combining \eqref{prod_est} and \eqref{deg_est}, we find
\begin{equation}  \label{fin_est}
 \int_{N \times S^\ell} |dv|_g^{n+\ell} dV_g  
 	\geq \vol(N)(\ell+1)\omega_{\ell+1}l^{(n+\ell)/2},
\end{equation}
with equality if and only if $|\nabla^M v|^2=0,$ and $|\nabla^{S^\ell}v|^2= const.,$ and $v(x, \cdot)$ is conformal.
It follows in this case that $v(x,\theta)= \tilde v (\theta)$
with $\tilde v \colon S^\ell \to S^\ell$ of degree $1.$
Observe, that $u \colon M \times S^\ell\to S^\ell$ realizes the equality in \eqref{fin_est}.
Therefore,
\begin{equation}
\inf_v \int_{N \times S^\ell} |dv|^{n+\ell} dV_g = \vol(N)(\ell+1)\omega_{\ell+1}l^{(n+\ell)/2},
\end{equation}
where the infimum is taken over all Lipschitz maps $v$ having the $l$-homotopy of $u.$
In particular, by the equality discussion above, minimizers need to be $(n+l)$-harmonic maps $v(x,\theta)= \tilde v (\theta)$,
with $|\nabla v|^2=const.$
Therefore, $\tilde v$ defines a harmonic selfmap of $S^\ell$ with constant density.
Since $\tilde v$ is non-trivial, it follows that $|\nabla \tilde v|^2 \geq \lambda_1(S^\ell)=l.$ 
Consequently, equality in \eqref{fin_est} is only achieved by maps of the form $A \circ u,$ with $A \in O(l+1).$
\end{proof}

Using \cref{prop_prod_est} instead of \cref{uniqueness}, assertion $(ii)$ of \cref{main}
follows along the same lines as assertion $(i.)$

\begin{proof}[Proof of \cref{mappingtori}]
Let $f \colon M \to S^1$ be a submersion.
Since $M$ is compact this is a proper submersion.
Moreover, $f$ has to be surjective, since otherwise $M$ would be contractible.
It follows by Ehresmann's lemma that $f \colon M \to S^1$ is a fibre bundle,
$F \to M \to S^1,$
with $F$ a smooth $(n-1)$-dimensional manifold.
As a consequence there is a diffeomorphism $\phi \colon F \to F,$
such that $M$ is obtained as the mapping torus corresponding to $\phi,$
i.e.\
\begin{align*}
M \cong \left( F \times [0,1] \right) / (x,0) \sim (\phi(x),1).
\end{align*}
Choose a metric $g_0$ on $F,$ which is invariant under $\phi.$
We claim that the metric $g_1=g_0 + dt^2$ defined on $F \times [0,1]$ descends to a smooth
metric $g$ on $M.$ 
Clearly, $g_1$ descends to a metric $g$ on $M,$ we only need to check that it is smooth.
This is clear near all points $(x,t)$ with $t \neq 0,1.$
We have coordinates with values in $F \times (- \eps, \eps)$ near the $t=0$-slice as 
follows.
\begin{align} \label{coordinates}
(x,t) \mapsto \begin{cases} (x,t-1) & \text{if}\ \ 	t \leq 1
							\\
							(\phi(x),t) & \text{if}\ \ 	t >0.
			\end{cases}
\end{align}
In these coordinates $g$ is given by $g_0 + dt^2,$ since $g_0$ is $\phi$-invariant.

It remains to show that $(M,g)$ admits an eigenmap.
Define $u \colon F \times [0,1] / (x,0) \sim (f(x),1) \to S^1 $ by $(x,t) \mapsto t.$
With respect to $g$ this is a Riemannian submersion.
Moreover, it follows from \eqref{coordinates} that $u$ has totally geodesic fibres.
Thus $u$ is an eigenmap.
\end{proof}

\end{document}